\documentclass[12pt]{amsart}
\usepackage{color}

\usepackage{amssymb}
\usepackage{verbatim}
\usepackage[toc,page]{appendix}
\usepackage{mathrsfs}

\newtheorem{thm}{Theorem}[section]

\newtheorem{lem}[thm]{Lemma}

\theoremstyle{definition}

\theoremstyle{remark}

\newtheorem{rem}{Remark}
\numberwithin{rem}{section}

\numberwithin{equation}{section}

\providecommand{\SL}{\operatorname{SL}}

\providecommand{\SL}{\operatorname{SL}}

\providecommand{\sym}{\operatorname{sym}}

\begin{document}

\title[Non-vanishing of Maass form $L$-functions ]{Non-vanishing of Maass form $L$-functions at the critical point}

\author[O. Balkanova]{Olga  Balkanova}
\address{Department of Mathematical Sciences, Chalmers University of Technology and the University of Gothenburg, SE-412 96 Gothenburg, Sweden}
\email{olgabalkanova@gmail.com}
\thanks{}

\author[B. Huang]{Bingrong Huang}
\address{School of Mathematical Sciences, Tel Aviv University, Tel Aviv, Israel}
\email{bingronghuangsdu@gmail.com}

\author[A. S\"{o}dergren]{Anders S\"{o}dergren}
\address{Department of Mathematical Sciences, Chalmers University of Technology and the University of Gothenburg, SE-412 96 Gothenburg, Sweden}
\email{andesod@chalmers.se}

\begin{abstract}
In this paper, we consider the family $\{L_j(s)\}_{j=1}^{\infty}$ of $L$-functions associated to an orthonormal basis $\{u_j\}_{j=1}^{\infty}$ of even Hecke-Maass forms for the modular group $\SL(2,\mathbb Z)$ with eigenvalues $\{\lambda_j=\kappa_{j}^{2}+1/4\}_{j=1}^{\infty}$. We prove the following effective non-vanishing result: At least $50 \%$ of the central values $L_j(1/2)$  with $\kappa_j \leq  T$ do not vanish as $T\rightarrow \infty$. Furthermore, we establish effective non-vanishing results in short intervals.
\end{abstract}

\thanks{The second author was supported  by the European Research Council, under the European Union's Seventh Framework Programme (FP7/2007-2013)/ERC grant agreement n$^{\text{o}}$~320755.
The third author was supported by a grant from the Swedish Research Council (grant 2016-03759).}

\maketitle

\tableofcontents


\section{Introduction}

The distribution of central values of automorphic $L$-functions is a well-studied topic in number theory with a number of important applications (see, e.g., the discussion in \cite{IS2}). In this paper we study the non-vanishing of central values of Maass form $L$-functions in the eigenvalue aspect. More precisely, we prove the following lower bound on the proportion of non-zero central $L$-values.
\begin{thm}\label{thm:main}
For any $\epsilon>0$ and for all sufficiently large $T$, we have
\begin{equation*}
\frac{\#\{u_j: L_j(1/2)\neq 0, \kappa_j\leq T \}}{\#\{u_j: \kappa_j\leq T \}} \geq \frac{1}{2}-\epsilon,
\end{equation*}
where $\{u_j\}_{j=1}^{\infty}$ is an orthonormal basis of even Hecke-Maass forms for the modular group $\SL(2,\mathbb Z)$ with eigenvalues $\{\lambda_j=\kappa_{j}^{2}+1/4\}_{j=1}^{\infty}$.
\end{thm}

The corresponding problem for $L$-functions associated to holomorphic cusp forms has a rich history. Starting from the work of Duke \cite{Duk}, many non-vanishing results have been established for central values of cusp form $L$-functions both in the level \cite{BF2, Dja, IS, KM, R, V}  and in the weight aspects \cite{BF, Fom, IS, LT, Luo3}. On the other hand, for the family of Maass form $L$-functions (in the eigenvalue aspect) similar questions have received increasing attention in the recent years. Alpoge and Miller \cite{AM} investigated the low-lying zeros for this family, while Xu \cite{Xu} and Liu \cite{Liu} studied non-vanishing problems in short intervals.

Theorem \ref{thm:main} is the first effective non-vanishing result for Maass form $L$-functions.  Making a comparison with the holomorphic case, this is an analogue of \cite[Theorem 5]{IS} by Iwaniec and Sarnak. Just as in the case of holomorphic forms the problem is closely connected to the existence problem for Landau-Siegel zeros and any improvement of the percentage $50\%$ will imply the nonexistence of Landau-Siegel zeros (see \cite{IS}).

It is also possible to establish effective non-vanishing results in short intervals. To this end, it is required to estimate the fourth moment of the Riemann zeta function in short intervals (see the proof of Theorem \ref{thm: second moment} for details). Unconditional results for this moment are only known for intervals of length at least $T^{2/3+\epsilon}$.
In other ranges, we combine recent results of Bourgain \cite{Bour} and Bourgain-Watt \cite{BW}, obtaining the following theorem.
\begin{thm}\label{thm:main2}
For any $\epsilon>0$ and $H=T^{\beta}$ with $\beta\leq 1$ we have
\begin{equation}
\sum_{\substack{T\leq \kappa_j\leq T+H\\ L_j(1/2)\neq 0}}1\geq \left( \frac{\Delta}{1+\Delta}-\epsilon\right)\sum_{T\leq \kappa_j\leq T+H}1,
\end{equation}
where the summations are restricted to even Hecke-Maass forms and
\begin{equation}
\Delta\leq\min(1-\alpha;1/4+3/4\beta),\quad \theta:=13/84,
\end{equation}
\begin{equation}
\alpha:=\begin{cases}
0,\quad \beta \geq 2/3+\epsilon;\\
2\theta, \quad 1273/4053+\epsilon\leq \beta \leq 2/3+\epsilon;\\
4\theta, \quad \beta\leq 1273/4053+\epsilon.
\end{cases}
\end{equation}
\end{thm}
This is an effective version of the recent result of Liu \cite{Liu}.

The methods used in this paper are quite classical and consist of a precise evaluation of mollified first and second moments of the central $L$-values in the family of Maass forms. An important first step is to find asymptotic formulas for the twisted first and second moments in this family and we combine results by Ivic \cite{Ivic2}, Ivic-Jutila \cite{IJ}, Jutila \cite{Jut1}, Kuznetsov \cite{Kuz}, Li \cite{Li} and Motohashi \cite{Mot1} in order to prove such formulas (see Theorem \ref{thm: first moment} and Theorem \ref{thm: second moment}). The second main step is to apply the mollification techniques developed by Kowalski and Michel \cite{K, KM, KM2} in order to prove asymptotic formulas for the mollified moments (see Sections \ref{sec:remove} and \ref{sec:moll}).

\section{Notation and preliminary results}

The letter $\epsilon$ denotes an arbitrarily small positive real number whose value may change from one line to another. All implied constants may depend on $\epsilon$ (if applicable), but this will always be supressed in the notation.

Let $\gamma$ be Euler's constant and let $\tau(n)$ denote the number of positive divisors of the integer $n$. We denote by $\Gamma(s)$ the gamma function,  by $S(m,n;c)$ the classical Kloosterman sum and by $\zeta(s)$ the Riemann zeta function. We have
\begin{equation}\label{standard estimate}
 |\zeta(1+it)^{-1}|\ll \log(1+|t|). 
\end{equation}

Let $\{u_j\}_{j=1}^{\infty}$  be an orthonormal basis in the space of Maass cusp forms for $\SL(2,\mathbb Z)$, consisting of common eigenfunctions of all Hecke operators, the hyperbolic Laplacian and the reflection operator $R$ defined by $(Rf)(z)=f(-\bar z)$. We denote by
$t_{j}(n)$ the eigenvalue of the $n$th Hecke operator acting on $u_{j}$, by $\lambda_{j}=1/4+\kappa_{j}^2$ the eigenvalue of the hyperbolic Laplacian acting on $u_{j}$ and by $\epsilon_j\in\{\pm1\}$ the eigenvalue of $R$ acting on $u_j$.
Every element of this basis admits a Fourier expansion of the form
\begin{equation*}
u_{j}(x+iy)=\sqrt{y}\sum_{n\neq 0}\rho_{j}(n)K_{i\kappa_j}(2\pi|n|y)\exp(2\pi inx),
\end{equation*}
where $K_{\alpha}(x)$ is the $K$-Bessel function and
\begin{equation*}
\rho_{j}(n)=\rho_{j}(1)t_{j}(n) \qquad (n\geq1).
\end{equation*}

For each $u_j$, we define the corresponding $L$-function by
\begin{equation*}
L_j(s)=\sum_{n=1}^{\infty}t_j(n)n^{-s}, \quad \Re{s}>1.
\end{equation*}
The completed $L$-function
\begin{equation*}
\Lambda_j(s)=\pi^{-s}\Gamma\left( \frac{s+\delta+i\kappa_j}{2}\right)\Gamma\left( \frac{s+\delta-i\kappa_j}{2}\right)L_j(s),
\end{equation*}
where $\delta=0$ if $\epsilon_j=1$ and $\delta=1$ if $\epsilon_j=-1$, can be analytically continued to the whole complex plane and satisfies the functional equation
\begin{equation}\label{eq:fuctional}
\Lambda_j(s)=\epsilon_j \Lambda_j(1-s),
\end{equation}
where $\epsilon_j=\pm1$ is called the sign of the functional equation.

It follows from the results of Katok and Sarnak \cite[Corollary 1]{KS} that $L_{j}(1/2)\geq 0$.
Moreover, if $\epsilon_j=-1$ then the functional equation \eqref{eq:fuctional} forces $L_j(1/2)$ to be zero. Therefore, we will restrict our attention to $L$-functions with even functional equations. Kuznetsov \cite[Theorem 2.11]{Kuz1} proved that the number of even ($\epsilon_j=+1$) and the number of odd ($\epsilon_j=-1$) Maass forms with $\kappa_j\leq T$ are asymptotically equal. Hence we have the following Weyl law for the even Maass cusp forms:
\begin{equation}\label{eq:weyl}
\#\{u_j: \kappa_j\leq T \text{ and } \epsilon_j=1\}\sim \frac{T^2}{24}.
\end{equation}
In addition, we mention that Ivic \cite{Ivic2} (see also \cite{Jut1}) has proved the subconvexity estimate
\begin{equation}\label{eq:subcJutila}
L_j(1/2)\ll \kappa_j^{1/3+\epsilon}
\end{equation}
for any $\epsilon>0$.

The symmetric square $L$-function associated to $u_j$ is defined by
\begin{equation*}
L_j(s, \sym^2)=\zeta(2s)\sum_{n=1}^{\infty}\frac{t_j(n^2)}{n^s}, \quad \Re{s}>1.
\end{equation*}
The corresponding completed $L$-function
\begin{equation*}
\Lambda_j(s,\sym^2)=\frac{L_j(s, \sym^2)}{\pi^{3s/2}}\Gamma\left(\frac{s}{2}\right)\Gamma\left(\frac{s+i\kappa_j}{2}\right)\Gamma\left(\frac{s-i\kappa_j}{2}\right)
\end{equation*}
can be analytically continued to the whole complex plane and satisfies the functional equation
\begin{equation*}
\Lambda_j(s,\sym^2)=\Lambda_j(1-s,\sym^2).
\end{equation*}
The $L$-functions $L_j(s, \sym^2)$ appear in the normalizing coefficients
\begin{equation}\label{alphaj}
\alpha_{j}=\frac{|\rho_{j}(1)|^2}{\cosh({\pi \kappa_j})}=\frac{2}{L_j(1, \sym^2)}
\end{equation}
in Kuznetsov's trace formula. According to \cite{HL, Iwa}, we have
\begin{equation}\label{eq:boundsalphaj}
\kappa_j^{-\epsilon}\ll \alpha_j\ll \kappa_j^{\epsilon}
\end{equation}
for any $\epsilon>0$.

\begin{lem}(The Kuznetsov trace formula)
Let $g=g(t)$ be an even function, holomorphic in $\Im{t}\leq 1/2+\epsilon$ and such that $g(t)\ll (|t|+1)^{-2-\epsilon}$ for some $\epsilon>0$.
For integers $m,n\geq 1$, we have
\begin{multline}\label{eq:Kuznetsovtraceformula}
\sum_{\epsilon_j=1}\alpha_jt_j(m)t_j(n)g(\kappa_j)=-\frac{1}{\pi}\int_{-\infty}^{\infty}\frac{\sigma_{2it}(m)\sigma_{2it}(n)}{(mn)^{it}|\zeta(1+2it)|^2}g(t)dt\\+
\frac{\delta_{m,n}}{2}\frac{1}{\pi^2}\int_{-\infty}^{\infty}t\tanh(\pi t)g(t)dt \\
+\frac{1}{2}\sum_{q=1}^{\infty}\frac{1}{q}\left[S(m,n;q)H^{+}\left(\frac{4\pi\sqrt{mn}}{q}\right) +S(m,-n;q)H^{-}\left(\frac{4\pi\sqrt{mn}}{q} \right)\right],
\end{multline}
where $\sigma_{\alpha}(n)=\sum_{d|n}d^{\alpha}$,
\begin{equation*}
H^{+}(x)=\frac{2i}{\pi}\int_{-\infty}^{\infty}\frac{r}{\cosh(\pi r)}J_{2ir}(x)g(r)dr,
\end{equation*}
\begin{equation*}
H^{-}(x)=\frac{4}{\pi^2}\int_{-\infty}^{\infty}\sinh(\pi r)K_{2ir}(x)rg(r)dr,
\end{equation*}
and $J_{\alpha}(x)$ and $K_{\alpha}(x)$ denote the $J$-Bessel and the $K$-Bessel functions.
\end{lem}
\begin{proof}
See, for example, \cite[Eq.\ 3.17]{CI}.
\end{proof}

We will also need the following approximate functional equation.

\begin{lem}(The approximate functional equation)
We have
\begin{equation}\label{eq:approxfuneq}
L_{j}(1/2)=2\sum_{n\geq 1}\frac{t_{j}(n)}{n^{1/2}}V(n,\kappa_j),
\end{equation}
where, for any fixed $\delta>0$ and $G(x)=\exp(-x^4)$,
\begin{equation}\label{eq:vnt}
V(n,t)=\frac{1}{2\pi i}\int_{(\delta)}n^{-x}G(x)\gamma(x,t)\frac{dx}{x}
\end{equation}
and
\begin{equation*}
\gamma(x,t)=\pi^{-x}\frac{\Gamma(1/4+x/2-it/2)\Gamma(1/4+x/2+it/2)}{\Gamma(1/4-it/2)\Gamma(1/4+it/2)}.
\end{equation*}
\end{lem}
\begin{proof}
See \cite[Lemma 3]{Liu}.
\end{proof}
Shifting the contour of integration in \eqref{eq:vnt} to $\Re{x}=A$, we obtain
\begin{equation}\label{eq:vntprop1}
V(n,t)\ll_A \left( \frac{t}{n} \right)^A
\end{equation}
for any $A>0$. Analogously, shifting the contour to $\Re{x}=-1/2+\epsilon$ for some $\epsilon>0$, yields
\begin{equation}\label{eq:vntprop2}
V(n,t)=1+O\left(\left(\frac{n}{t}\right)^{1/2-\epsilon}\right).
\end{equation}

\section{Choice of weight function}

Similarly to \cite[Eqs.\ (2.14), (3.2)]{IJ}, we define
\begin{multline}\label{eq:h}
h(r)=h(r,K,G)\\ =\frac{r^2+1/4}{r^2+1000}\left[\exp\left(-\left(\frac{r-K}{G}\right)^2\right)+\exp\left(-\left(\frac{r+K}{G}\right)^2 \right)\right]
\end{multline}
and
\begin{equation*}
\Omega(r)=\frac{1}{\sqrt{\pi} G}\int_{T}^{T+H}h(r,K,G)\,dK.
\end{equation*}
For an arbitrary $A>1$ and some $c>0$ we have (see \cite{IJ})
\begin{equation}\label{omega1}
\Omega(r)=1+O(r^{-A})\text{ if } T+cG\sqrt{\log T}<r<T+H-cG\sqrt{\log T},
\end{equation}
\begin{equation}\label{omega2}
\Omega(r)=O((|r|+T)^{-A})\text{ if }
r<T-cG\sqrt{\log T}\text{ or } r>T+H+cG\sqrt{\log T},
\end{equation}
and otherwise
\begin{equation}\label{omega3}
\Omega(r)=1+O(G^3(G+\min(|r-T|,|r-T-H|))^{-3}).
\end{equation}
In the next two sections we use the above functions, together with the Kuznetsov trace formula, to compute the smoothed moments
\begin{equation*}
M_n(l):=\frac{1}{\sqrt{\pi} G}\int_{T}^{T+H}M_n(l;h)\,dK, \quad n=1,2,
\end{equation*}
where
\begin{equation*}
M_n(l;h):=\sum_{\epsilon_j=1}\alpha_j t_j(l)L_{j}^{n}(1/2)h(\kappa_j).
\end{equation*}
\begin{lem}\label{lem:removingsmoothing}
Assume that $G\leq H/\log{T}$ and that for any $\delta>0$ we have
 \begin{equation}
\sum_{\substack{\epsilon_j=1\\L_j(1/2)\neq 0}}\Omega(\kappa_j)\geq (PN-\delta)\sum_{\epsilon_j=1}\Omega(\kappa_j),
\end{equation}
where $PN$ is a lower bound on the proportion of non-vanishing central $L$-values.
Then there exists $\delta_1=\delta_1(\delta)$ such that
 \begin{equation}
\sum_{\substack{T\leq \kappa_j\leq T+H\\\epsilon_j=1;L_j(1/2)\neq 0}}1\geq (PN-\delta_1)\sum_{\substack{T\leq \kappa_j\leq T+H\\\epsilon_j=1}}1.
\end{equation}
\end{lem}
\begin{proof}
Using the properties of the function $\Omega(r)$ and Weyl's law we infer
\begin{multline}
\sum_{\substack{T\leq \kappa_j\leq T+H\\\epsilon_j=1;L_j(1/2)\neq 0}}1\geq
\sum_{\substack{\epsilon_j=1\\L_j(1/2)\neq 0}}\Omega(\kappa_j)-\sum_{\substack{\kappa_j<T\text{ or }\kappa_j>T+H\\L_j(1/2)\neq 0}}\Omega(\kappa_j)\\
=\sum_{\substack{\epsilon_j=1\\L_j(1/2)\neq 0}}\Omega(\kappa_j)+O\left(TG\sqrt{\log{T}} \right)
\\=(PN-\delta)\sum_{\substack{T\leq \kappa_j\leq T+H\\\epsilon_j=1}}1+O\left(TG\sqrt{\log{T}} \right)
\geq(PN-2\delta)\sum_{\substack{T\leq \kappa_j\leq T+H\\\epsilon_j=1}}1
\end{multline}
provided that
$G\leq H/\log{T}.$
\end{proof}

\section{The twisted first moment}
The first moment $M_1(l)$ for $l=1$ and $H=T$ was  evaluated asymptotically by Ivic and Jutila in \cite{IJ}. The case $l>1$ can be analyzed similarly.
However, for our purposes it is more convenient to use an approximate functional equation of the form \eqref{eq:approxfuneq} together with the Kuznetsov trace formula \eqref{eq:Kuznetsovtraceformula} restricted to even Maass forms.

\begin{thm}\label{thm: first moment}
For any $\epsilon>0$, any $ G, H\ll T$ and any $l\ll T^{1-\epsilon}$, we have
\begin{equation*}
M_1(l)=\frac{2}{\pi^2\sqrt{l}}\int_{T}^{T+H}K\,dK+ O\left((H+G)T^{1/2+\epsilon}+\frac{HT}{\sqrt{l}} \left( \frac{l}{T}\right)^{1/2-\epsilon}\right).
\end{equation*}
\end{thm}
\begin{proof}
The approximate functional equation \eqref{eq:approxfuneq} yields
\begin{equation*}
M_1(l;h)=2\sum_{n\geq 1}\frac{1}{n^{1/2}}\sum_{\epsilon_j=1}\alpha_j t_j(l)t_j(n)V(n,\kappa_j)h(\kappa_j).
\end{equation*}
At the cost of a negligible error term, we use \eqref{eq:vntprop1} to truncate the summation over $n$ to the range $n\ll T^{1+\epsilon/2}$.

Next, applying the Kuznetsov trace formula with $g(t):=V(n,t)h(t)$, we get
\begin{equation*}
M_1(l;h)=\sum_{i=1}^{4}F_{i}(l;h),
\end{equation*}
where
\begin{equation*}
F_1(l;h)=\frac{1}{\pi^2\sqrt{l}}\int_{-\infty}^{\infty}t\tanh(\pi t)V(l,t)h(t)\,dt,
\end{equation*}
\begin{equation*}
F_2(l;h)=-\frac{2}{\pi}\int_{-\infty}^{\infty}\frac{\sigma_{2it}(l)}{l^{it}}\sum_{n\ll T^{1+\epsilon/2}}\frac{\sigma_{2it}(n)}{n^{1/2+it}}\frac{V(n,t)h(t)}{|\zeta(1+2it)|^2}\,dt,
\end{equation*}
\begin{equation*}
F_3(l;h)=\sum_{n\ll T^{1+\epsilon/2}}\sum_{q\geq 1}\frac{S(l,n;q)}{\sqrt{n}q}H^{+}\left( \frac{4\pi \sqrt{ln}}{q}\right)
\end{equation*}
and
\begin{equation*}
F_4(l;h)=\sum_{n\ll T^{1+\epsilon/2}}\sum_{q\geq 1}\frac{S(l,-n;q)}{\sqrt{n}q}H^{-}\left( \frac{4\pi \sqrt{ln}}{q}\right).
\end{equation*}
The summand $F_1(l;h)$ contains the main term. More precisely, using \eqref{eq:vntprop2} we find that
\begin{equation*}
F_1(l;h)=\frac{2}{\pi^{3/2}\sqrt{l}}GK+O\left(\frac{GK}{\sqrt{l}}\left(\frac{l}{K}\right)^{1/2-\epsilon}\right).
\end{equation*}

Applying \eqref{standard estimate}, \cite[Eq.\ (3.5)]{IJ} and estimating the summation over $n$ trivially, the contribution of the continuous spectrum can be bounded as follows:
\begin{equation*}
\frac{1}{\sqrt{\pi}G}\int_{T}^{T+H}F_2(l;h)\,dK\ll (H+G)T^{1/2+\epsilon}.
\end{equation*}
Finally, we estimate the contribution of  $F_3(l;h)$ and $F_4(l;h)$ for $l\ll T^{1-\epsilon}$. Similarly to \cite[p.\ 7]{IJ}, the summation over $q$ can be truncated to the range $q\ll T^B$, for some $B>1$, at the cost of a negligible error term. Then it remains to analyze the integral transforms $H^{\pm}(x)$ for $x=4\pi\sqrt{ln}/q\ll T^{1-\epsilon/4}$. This has been done in Li's paper \cite[Sections 4 and 5]{Li}.
In particular,  Propositions $4.1$ and $5.1$ of \cite{Li} imply that for $x\ll T^{1-\epsilon/4}$ and for any fixed $A>0$ we have
$H^{\pm}(x)\ll T^{-A}$. Therefore, the contributions of  $F_3(l;h)$ and $F_4(l;h)$ are negligibly small.
\end{proof}

\section{The twisted second moment}

In this section we prove the following asymptotic formula.
\begin{thm}\label{thm: second moment} For any $\epsilon>0$, $l\geq 1$, $T^{\epsilon}\ll H\leq T$, $T^{\epsilon}\ll G\ll T^{1-\epsilon}$
the following asymptotic formula holds
\begin{multline}\label{second moment asymp}
M_2(l)=\frac{4}{\pi^2}\frac{\tau(l)}{\sqrt{l}}\int_{T}^{T+H}K\big(\gamma-\log(2\pi\sqrt{l})+\log{K}\big)\,dK\\
+O\left((lT)^{\epsilon}\left[\frac{G^2H}{l^{1/2}T}+T^{\alpha}(H+G)+\frac{l^{3/2}H}{TG^2}+\left(\frac{lT}{H}\right)^{1/2}+\left(\frac{lT}{G}\right)^{1/2}\right]\right),
\end{multline}
where for $\theta=13/84$ we define
\begin{equation}\label{eq:choiceofalpha}
\alpha:=\begin{cases}
0,\quad H\gg T^{2/3+\epsilon};\\
2\theta, \quad T^{1273/4053+\epsilon}\ll H\ll T^{2/3+\epsilon};\\
4\theta, \quad H\ll T^{1273/4053+\epsilon}.
\end{cases}
\end{equation}
\end{thm}
\begin{rem}\label{rem:sm}
Note that the error term in \eqref{second moment asymp} is less than the main term as long as
$$l\ll T^{-\epsilon}\min \left(T^{2-2\alpha}, TG, \frac{H^2}{G^2}T^{2-2\alpha}, H^{3/2}T^{1/2}, HT^{1/2}G^{1/2}\right).$$
In view of restrictions in Lemma \ref{lem:removingsmoothing}, the optimal choice of $G$ is $G=H/\log{T}$. Taking $H:=T^{\beta}$, we finally obtain
\begin{equation}\label{eq:restrictiononl}
l\ll T^{-\epsilon}\min \left(T^{2-2\alpha},T^{1/2+3/2\beta}\right).
\end{equation}
\end{rem}
In order to prove Theorem \ref{thm: second moment} we apply an explicit formula for the second twisted moment. Let
\begin{equation*}
\hat{h}(s)=\int_{-\infty}^{\infty}rh(r)\frac{\Gamma(s+ir)}{\Gamma(1-s+ir)}\,dr
\end{equation*}
and note that $\hat{h}(s)$ is an entire function.
Furthermore, for $-3/2<c<1/2$, we define the special functions
\begin{equation*}
\Psi^{+}(x;h)=\int_{(c)}\Gamma(1/2-s)^2\tan(\pi s)\hat{h}(s)x^s\,ds
\end{equation*}
and
\begin{equation*}
\Psi^{-}(x;h)=\int_{(c)}\Gamma(1/2-s)^2\frac{\hat{h}(s)}{\cos(\pi s)}x^s\,ds.
\end{equation*}

\begin{lem}\label{lem:m2lh}
Let $l\geq1$. Then we have
\begin{equation}\label{m2lh}
M_2(l;h)=\sum_{v=1}^{7}R_{v}(l;h),
\end{equation}
where
\begin{equation*}
R_1(l;h)=\frac{2}{\pi^{3}i}\frac{\tau(l)}{\sqrt{l}}\left[\big(\gamma -\log(2\pi \sqrt{l})\big)(\hat{h})'(1/2)+1/4(\hat{h})''(1/2)\right],
\end{equation*}

\begin{equation*}
R_2(l;h)=\frac{1}{\pi^{3}}\sum_{m=1}^{\infty}\frac{\tau(m)\tau(m+l)}{\sqrt{m}}\Psi^{+}(m/l;h),
\end{equation*}

\begin{equation*}
R_3(l;h)=\frac{1}{\pi^{3}}\sum_{m=1}^{\infty}\frac{\tau(m)\tau(m+l)}{\sqrt{m+l}}\Psi^{-}(1+m/l;h),
\end{equation*}

\begin{equation*}
R_4(l;h)=\frac{1}{\pi^{3}}\sum_{m=1}^{l-1}\frac{\tau(m)\tau(l-m)}{\sqrt{m}}\Psi^{-}(m/l;h),
\end{equation*}

\begin{equation*}
R_5(l;h)=-\frac{1}{2\pi^{3}}\frac{\tau(l)}{\sqrt{l}}\Psi^{-}(1;h),
\end{equation*}

\begin{equation*}
R_6(l;h)=-\frac{12i}{\pi^{2}}\sigma_{-1}(l)\sqrt{l}h'(-i/2),
\end{equation*}

\begin{equation*}
R_7(l;h)=-\frac{1}{\pi}\int_{-\infty}^{\infty}\frac{|\zeta(1/2+ir)|^4}{|\zeta(1+2ir)|^2}\frac{\sigma_{2ir}(l)}{l^{ir}}h(r)\,dr.
\end{equation*}
\end{lem}

\begin{proof}
See, for example, \cite[Lemma 3.8]{Mot1} or \cite[p. 457]{Ivic2}.
\end{proof}

Note that our choice of  weight function $h(\kappa_j)$ in \eqref{eq:h} is slightly different from the one of \cite{Ivic2, Jut1, Mot1}. In particular, in the argument below it is required to divide all estimates from \cite{Ivic2, Jut1, Mot1} by $K^2$.

\begin{proof}[Proof of Theorem \ref{thm: second moment}]
The main term in the asymptotic formula for the second moment is given by $R_1(l;h)$. Using \cite[Eqs.\ 3.3.37, 3.3.38 and p.\ 129]{Mot1}, we find that
\begin{equation}\label{eq:H1}
R_1(l;h)=4\frac{KG}{\pi^{3/2}}\frac{\tau(l)}{\sqrt{l}}\big(\gamma-\log{(2\pi\sqrt{l})}+\log{K}\big)+O\left(\frac{G^3K^{\epsilon}}{Kl^{1/2}}\right).
\end{equation}

Next, according to \cite[p.\ 458]{Ivic2} we have
\begin{equation}\label{eq:H346}
\frac{1}{\sqrt{\pi}G}\int_{T}^{T+H}\big(R_3(l;h)+R_5(l;h)+R_6(l;h)\big)\,dK\ll \frac{l^{1/2}}{G}(lT)^{\epsilon}.
\end{equation}
Applying the standard estimate \eqref{standard estimate},
together with an estimate for the fourth moment of the Riemann zeta function in short intervals (see \cite[p. 14]{Ivic3}), we also show that for $H\gg T^{2/3+\epsilon}$
\begin{equation}\label{eq:H7}
\frac{1}{\sqrt{\pi}G}\int_{T}^{T+H}R_7(l;h)\,dK\ll (H+G)(lT)^{\epsilon}.
\end{equation}
In order to estimate the contribution of $R_7(l;h)$ in the range $$T^{1273/4053+\epsilon}\ll H\ll T^{2/3+\epsilon},$$ we combine the estimate of Bourgain-Watt for the second moment of the Riemann zeta function in short intervals \cite[Theorem 3]{BW} and the subconvexity estimate of Bourgain \cite[Theorem 5]{Bour}
\begin{equation}\label{subc:bourgain}
|\zeta(1/2+it)|\ll |t|^{\theta+\epsilon}, \quad \theta=13/84,
\end{equation}
getting
\begin{equation}\label{eq:H72}
\frac{1}{\sqrt{\pi}G}\int_{T}^{T+H}R_7(l;h)\,dK\ll T^{2\theta}(H+G)(lT)^{\epsilon}.
\end{equation}
Finally, for $H\ll T^{1273/4053+\epsilon}$ we use solely \eqref{subc:bourgain} to show that
\begin{equation}\label{eq:H73}
\frac{1}{\sqrt{\pi}G}\int_{T}^{T+H}R_7(l;h)\,dK\ll T^{4\theta}(H+G)(lT)^{\epsilon}.
\end{equation}
Note that it follows from estimates in \cite[Section 4]{Ivic2} that
\begin{equation*}
\Psi^{-}(m/l;h)\ll \frac{1}{K^2}\left[ \left( \frac{m}{l}\right)^{-3/2+\epsilon}KG^{-1+\epsilon}+\left(\frac{l}{m} \right)^{3/2}GK^{-1+\epsilon}\right].
\end{equation*}
Therefore,
\begin{multline}\label{eq:H4}
\frac{1}{\sqrt{\pi}G}\int_{T}^{T+H}R_4(l;h)\,dK\\ \ll\frac{l^{\epsilon}}{G}\int_{T}^{T+H}\sum_{m=1}^{l-1}\frac{1}{\sqrt{m}K^2}\left[ \left( \frac{m}{l}\right)^{-3/2+\epsilon}KG^{-1+\epsilon}+\left(\frac{l}{m} \right)^{3/2}GK^{-1+\epsilon}\right]dK\\
\ll \frac{l^{3/2+\epsilon}}{G}\int_{T}^{T+H}\left( \frac{G^{\epsilon}}{KG}+\frac{GK^{\epsilon}}{K^3}\right)dK\ll l^{3/2}(lT)^{\epsilon}\left[\frac{H}{TG^2}+\frac{H}{T^3}\right].
\end{multline}

It remains to consider $R_2(l;h)$. According to \cite[Eqs.\ 3.4.20, 3.4.21]{Mot1}, the contribution of $R_2(l;h)$ is negligibly small for all terms with $G^2m/l\gg G^{\epsilon}$.
In order to estimate the contribution of $R_2(l;h)$ for $G^2m/l\ll G^{\epsilon}$, we use the asymptotic formula for $\Psi^{+}(m/l;h)$ proved by Jutila in \cite[p.\ 173]{Jut1}. According to this formula, the main term of $\Psi^{+}(m/l;h) $ is, up to a constant factor, equal to
\begin{multline*}
GK^{1/2}\left(\frac{m}{l}\right)^{1/4}\left(\sqrt{\frac{m}{l}}+\sqrt{1+\frac{m}{l}}
 \right)^{-2iK}\\
\times \exp\left( -G^2\log^2\left[\sqrt{\frac{m}{l}}+\sqrt{1+\frac{m}{l}}\right]\right).
\end{multline*}
We continue by estimating the integral over $K$ in two different ways: first trivially and then using \cite[Lemma 4.3]{T}. More precisely,
\begin{multline*}
\int_{T}^{T+H}K^{1/2}\exp\left(-2iK \log\left(\sqrt{\frac{m}{l}}+\sqrt{1+\frac{m}{l}}\right)\right)dK\\ \ll
\min\left( HT^{1/2},\frac{T^{1/2}}{\log\left(\sqrt{\frac{m}{l}}+\sqrt{1+\frac{m}{l}}\right)} \right)
\ll T^{1/2}\min\left(H, \sqrt{\frac{l}{m}} \right).
\end{multline*}
Consequently,
\begin{multline}\label{eq:H2}
\frac{1}{\sqrt{\pi}G}\int_{T}^{T+H}R_2(l;h)\,dK\ll \sum_{m<lG^{\epsilon-2}}\frac{l^{\epsilon}}{\sqrt{m}}\left(\frac{m}{l}\right)^{1/4}T^{1/2}\min\left(H, \sqrt{\frac{l}{m}} \right)\\
\ll \sum_{m<l/H^2}\frac{HT^{1/2}l^{\epsilon}}{m^{1/4}l^{1/4}}+\sum_{l/H^2\leq m<lG^{\epsilon-2}}\frac{T^{1/2}l^{1/4+\epsilon}}{m^{3/4}}\ll (lT)^{1/2+\epsilon}\left(\frac{1}{\sqrt{H}}+\frac{1}{\sqrt{G}}\right).
\end{multline}

Finally, combining \eqref{eq:H1}, \eqref{eq:H346}, \eqref{eq:H7}, \eqref{eq:H72}, \eqref{eq:H73}, \eqref{eq:H4} and \eqref{eq:H2} completes the proof of Theorem \ref{thm: second moment}.
\end{proof}

\section{Removing the harmonic weight}\label{sec:remove}

In this section, we show that the harmonic weight $\alpha_j=2/L_j(1,\sym^2)$ can be removed at the cost of a negligible error term. With this goal, we apply the techniques of \cite{KM, LW}.

Consider the Dirichlet series
\begin{equation}\label{eq:omegax}
\omega_j(x):=\sum_{n\leq x}\frac{t_j(n^2)}{n},\quad \omega_j(x,y):=\sum_{x < n\leq y}\frac{t_j(n^2)}{n}.
\end{equation}
Recall that we use $\epsilon$ to denote a small positive number, which might change from line to line.

\begin{lem}\label{lem:appwxy}
Let $x=T^{\epsilon}$ and let $y=T^{a}$ for some large $a>0$. Then
there exists a large $b=b(a)>0$ such that
\begin{equation*}
L_j(1,\sym^2)=\zeta(2)\omega_j(x)+\zeta(2)\omega_j(x,y)+O(T^{-b}).
\end{equation*}
\end{lem}
\begin{proof}
Similarly to \cite[Lemma 2.3]{LW}, we apply Perron's formula (see \cite[Corollary 2.1, p. 133]{Ten}) with $B(n)=n^{2\theta+\epsilon}$ and $\theta=7/64$. Then, for $y,z\geq 2$, the following asymptotic formula holds:
\begin{multline*}
\sum_{n\leq y}\frac{t_j(n^2)}{n}=\frac{1}{2\pi i}\int_{1/\log{y}-iz}^{1/\log{y}+iz}\frac{L_j(1+s,\sym^2)}{\zeta(2s+2)}\frac{y^s}{s}ds\\+O\left( \frac{y^{\epsilon}}{z}+\frac{y^{2\theta+\epsilon}}{y}\left(1+\frac{y}{z}\log{z}\right)\right).
\end{multline*}
Moving the contour of integration to $\Re{s}=-1/2+\epsilon$, for some $\epsilon>0$, and applying the convexity estimate
\begin{equation*}
L_j(\sigma+it,\sym^2)\ll t^{(1-\sigma)/2+\epsilon}|t^2-\kappa_{j}^{2}|^{(1-\sigma)/2+\epsilon},
\end{equation*}
we obtain
\begin{equation}\label{eq:approxlsym2}
\sum_{n\leq y}\frac{t_j(n^2)}{n}=\frac{L_j(1,\sym^2)}{\zeta(2)}+O\left((yz)^{\epsilon}\left[\frac{y^{2\theta}}{y}+\frac{y^{2\theta}}{z}+\frac{z^{3/4}}{y^{1/2}}\right]\right).
\end{equation}
The optimal choice of $z$, namely $z=y^{4/7(1/2+2\theta)}$,  yields the error term $O(y^{-2/7+6\theta/7+\epsilon})$ in \eqref{eq:approxlsym2}. Finally, the statement follows by choosing $x=T^{\epsilon}$ and $y=T^{a}$.
\end{proof}

\begin{lem}\label{wxyest}
For any $\epsilon>0$, $H \leq T$, $A>60$ and integer $i\geq 1$, we have
\begin{equation*}
\sum_{j}\Omega(\kappa_j)\omega_j(x,y)^{2i}\ll T^{\epsilon}
\end{equation*}
provided that $T^{60}\leq x^i\leq y^i\leq T^A$.
\end{lem}
\begin{proof}
Choosing coefficients $c_n\ll n^{-1+\epsilon}$ that are non-zero only in the interval $[N,2N]$, the large sieve inequality of Luo \cite[Eq.\ (34)]{Luo1} takes the form
\begin{equation}\label{luolargesieve}
\sum_{j}\Omega(\kappa_j)\left| \sum_{N\leq n\leq 2N}c_nt_j(n^2)\right|^2\ll (TN)^{\epsilon}(1+T^5N^{-1/12}).
\end{equation}
Repeating the arguments of \cite[Lemma 2.5 (see also Lemma 2.4)]{LW} and using \eqref{luolargesieve}, we obtain
\begin{equation*}
\sum_{j}\Omega(\kappa_j)\omega_j(x,y)^{2i}\ll T^{\epsilon}\sum_{j\ll \log{T}}2^{-j}\left( 1+T^5(x^i2^{-j-1})^{-1/12}\right).
\end{equation*}
This is bounded by $T^{\epsilon}$ provided  that $x^i\geq T^{60}$.
\end{proof}

\begin{lem}\label{lem:wxy}
Let $s_j$ be a sequence of complex numbers such that for $\epsilon>0$ and a sufficiently large constant $A>0$ and $H\leq T$, we have
\begin{equation}\label{eq:cond1}
\sum_{j}\Omega(\kappa_j)\alpha_j|s_j|\ll HT^{1+\epsilon}
\end{equation}
and
\begin{equation}\label{eq:cond2}
\max_{T\leq \kappa_j\leq T+H}|\alpha_js_j|\ll HT^{1-A\epsilon}.
\end{equation}
Then there exists $\delta=\delta(\epsilon,A)>0$ such that for $x=T^{\epsilon_0}$, we have
\begin{equation*}
\sum_{j}\Omega(\kappa_j)s_j=\frac{\zeta(2)}{2}\sum_{j}\Omega(\kappa_j)\alpha_js_j\omega_j(x)+O(HT^{1-\delta}).
\end{equation*}
\end{lem}

\begin{proof}
It follows from the definition of $\alpha_j$ (see \eqref{alphaj}) that
\begin{equation*}
\sum_{j}\Omega(\kappa_j)s_j=\frac{1}{2}\sum_{j}\Omega(\kappa_j)\alpha_js_jL_j(1,\sym^2).
\end{equation*}
We first approximate $L_j(1,\sym^2)$ by applying Lemma \ref{lem:appwxy}. Then, using H\"{o}lder's inequality with $1/(2p)+1/q=1$ together with Lemma \ref{wxyest} and conditions \eqref{eq:cond1} and \eqref{eq:cond2},  we show that
\begin{multline}
\sum_{j}\Omega(\kappa_j)\alpha_j s_j\omega_j(x,y)\ll \left( \sum_{j}\Omega(\kappa_j)|\omega_j(x,y)|^{2p}\right)^{1/(2p)} \left(\sum_{j}\Omega(\kappa_j)|\alpha_j s_j|^{q}\right)^{1/q}\\
\ll T^{\epsilon/(2p)}\big(\max_{T\leq \kappa_j\leq T+H}|\alpha_js_j|\big)^{(q-1)/q}\left(\sum_{j}\Omega(\kappa_j)|\alpha_j s_j|\right)^{1/q}\ll HT^{1-\delta}
\end{multline}
for some $\delta>0$, provided that  $p$ and $A$ are chosen such that $120/\epsilon_0<2p<A$. This concludes the proof.
\end{proof}

\section{Mollification and non-vanishing}\label{sec:moll}

In order to determine a proportion of non-vanishing for the central $L$-values, we apply the method of Kowalski and Michel developed in \cite{K, KM, KM2}. In particular, we choose mollifiers of the form
\begin{equation}\label{eq:mollifier}
\mathcal{M}(j)=\sum_{m\leq T^{\Delta}}\frac{x_m}{\sqrt{m}}t_{j}(m),
\end{equation}
where $\Delta$ is called the logarithmic length of the mollifier and $\{x_m\}_{m=1}^{\infty}$ is a sequence of real numbers supported on squarefree numbers $m\leq T^{\Delta}$ and satisfying $x_m\ll m^{\epsilon}$ (see  \cite[Section 5]{KM2}). Similarly to \cite[Lemma 8.1]{BF}, it can be shown that for any $\Delta<1-\epsilon$ and $H\leq T$, we have
\begin{equation}\label{eq:squaremollifier}
\sum_{j}\Omega(\kappa_j)\alpha_j \mathcal{M}^2(j)\ll HT^{1+\epsilon}.
\end{equation}

Our goal is to compute the mollified moments
\begin{equation*}
\mathfrak{M}_n(T):=\sum_{\epsilon_j=1}\Omega(\kappa_j)\big(\mathcal{M}(j)L_{j}(1/2)\big)^n \qquad\text { for }n=1,2.
\end{equation*}
In order to apply Lemma \ref{lem:wxy}, we need to verify that the conditions \eqref{eq:cond1} and \eqref{eq:cond2} with $s_j=(\mathcal{M}(j)L_{j}(1/2))^n$ hold for $n=1,2$. Estimating \eqref{eq:mollifier} trivially and applying  \eqref{eq:subcJutila} and \eqref{eq:boundsalphaj}, we find that \eqref{eq:cond2} is satisfied for any
\begin{equation}\label{rest1}\Delta \leq \frac{1/3+\beta-A\epsilon}{1+2d},
\end{equation}
where $d:=7/64$ and $H=T^{\beta}$. Next, in order to verify the condition \eqref{eq:cond1} for the first moment, we apply the Cauchy-Schwarz inequality and use \eqref{eq:squaremollifier} and Theorem \ref{thm: second moment}. Finally, following the arguments of \cite[Section 5.2]{KM2} (see also \cite[Lemma 8.5]{BF}) and using Theorem \ref{thm: second moment} and Remark \ref{rem:sm}, we show that \eqref{eq:cond1} is satisfied also for $s_j=(\mathcal{M}(j)L_{j}(1/2))^2$ as long as \begin{equation}\label{rest2}\Delta\leq\min(1-\alpha;1/4+3/4\beta),
\end{equation}
where $\alpha$ is defined by \eqref{eq:choiceofalpha} and $H=T^{\beta}$.
We remark that for any $\beta\geq 0$ we have
$$1/4+3/4\beta<\frac{1/3+\beta-A\epsilon}{1+2d}.$$
Therefore, the condition \eqref{rest1} is contained in \eqref{rest2}.
We conclude that, for $x=T^{\epsilon}$ and some $\delta>0$,
\begin{equation*}
\mathfrak{M}_n(T)=\frac{\zeta(2)}{2}\sum_{\epsilon_j=1}\Omega(\kappa_j)\alpha_j \mathcal{M}^n(j)L_{j}^{n}(1/2)\omega_j(x)+O(HT^{1-\delta}).
\end{equation*}

The next step is to evaluate $\mathfrak{M}_n(T)$ for $n=1,2$. Using \eqref{eq:mollifier} and  \eqref{eq:omegax}, we find that the second moment is equal to
\begin{multline*}
\mathfrak{M}_2(T)=\frac{\zeta(2)}{2}\sum_{d\leq x}\frac{1}{d}\sum_{b}\frac{1}{b}\sum_{m_1,m_2}\frac{x_{m_1b}x_{m_2b}}{\sqrt{m_1m_2}}\\
\times \sum_{r|(d^2,m_1m_2)}M_2\left(\frac{d^2m_1m_2}{r^2}\right)
+O(HT^{1-\delta}).
\end{multline*}
According to  Theorem \ref{thm: second moment}, the error term in the asymptotic formula for $M_2(d^2m_1m_2/r^2)$ contributes $O(HT^{1-\delta})$ to $\mathfrak{M}_2(T)$ for any  $\Delta$ satisfying  \eqref{rest2} and some $\delta>0$. To evaluate the contribution of the main term, we follow the arguments of \cite[Section 5]{KM2}. Consequently,
for any $\Delta$ satisfying  \eqref{rest2}, we have
\begin{multline*}
\mathfrak{M}_2(T)=\frac{\zeta(2)^3}{\pi^2}\int_{T}^{T+H}K\frac{\log{K^2}+2\log{T^{\Delta}}}{\log{T^{\Delta}}}\,dK+O(HT^{1-\delta})\\
=\frac{\zeta(2)^3}{\pi^2}(2TH+H^2)\frac{1+\Delta}{\Delta}+O(HT^{1-\delta}).
\end{multline*}
Similarly, for any $\Delta<1-\epsilon$, we obtain
\begin{equation*}
\mathfrak{M}_1(T)=\frac{\zeta(2)^2}{\pi^2}\int_{T}^{T+H}K\,dK+O(HT^{1-\delta})=\frac{\zeta(2)^2}{2\pi^2}\left(2TH+H^2\right)+O(HT^{1-\delta}).
\end{equation*}
Finally, we are ready to finish the proof of Theorems \ref{thm:main} and \ref{thm:main2}.

\begin{proof}[Proof of Theorem \ref{thm:main2}]
Using \eqref{eq:weyl}, we find that
\begin{equation}
\#\{u_j: T\leq \kappa_j\leq T+H \text{ and } \epsilon_j=1\}\sim \frac{2TH+H^2}{24}.
\end{equation}
As a consequence of the Cauchy-Schwarz inequality, for any $\Delta$  satisfying \eqref{rest2}, we have
\begin{equation*}
\frac{24}{2TH+H^2}\sum_{L_{j}(1/2)\neq 0}\Omega(\kappa_j)\geq \frac{24}{2TH+H^2} \frac{\mathfrak{M}_{1}^{2}(T)}{\mathfrak{M}_2(T)}\geq \frac{\Delta}{1+\Delta}-\epsilon.
\end{equation*}
Applying Lemma \ref{lem:removingsmoothing}, we prove Theorem \ref{thm:main2}.
\end{proof}
\begin{proof}[Proof of Theorem \ref{thm:main}]
Taking $H=T$ and using a dyadic decomposition yields that for any $\Delta<1-\epsilon$, we have
\begin{equation*}
\frac{24}{T^2}\sum_{\substack{\kappa_j\leq T\\L_{j}(1/2)\neq 0}}1\geq \frac{\Delta}{1+\Delta}-\epsilon.
\end{equation*}
Theorem \ref{thm:main} follows by taking the largest admissible $\Delta$ and using the asymptotic formula \eqref{eq:weyl}.
\end{proof}


\nocite{}

\end{document}